\pdfoutput=1
\documentclass[reqno]{amsart}

\usepackage{hyperref}
\usepackage{amsmath}
\usepackage{mathrsfs}
\usepackage{amsthm}
\usepackage{amssymb}
\usepackage{fullpage}
\usepackage[T1]{fontenc}
\usepackage{lmodern}
\usepackage[utf8]{inputenc}

\usepackage[backend=bibtex,safeinputenc=true,style=alphabetic,firstinits=true,useprefix=true,url=false,maxnames=6,doi=false,isbn=false,url=false]{biblatex}

\def\woMR#1{\w@MR#1MR#1MR\relax}%
\def\w@MR#1MR#2MR#3\relax{#2}
\def\MR@URL#1 #2\relax{http://www.ams.org/mathscinet-getitem?mr=#1}
\def\MRnumber#1{\href{\MR@URL#1 \relax}{\nolinkurl{\woMR#1}}}% make a link to mathscinet
\DeclareFieldFormat{mrnumber}{\mkbibacro{MR}\addcolon\space\MRnumber{#1}}
\newbibmacro*{doi+eprint+url}{%
\iftoggle{bbx:doi}
{\printfield{doi}}
{}%
\newunit\newblock
\iftoggle{bbx:eprint}
{\usebibmacro{eprint}}
{}%
\newunit\newblock
\iftoggle{bbx:url}
{\usebibmacro{url+urldate}}
{}%
\newunit\newblock
\printfield{mrnumber}}

\newbibmacro{string+urldoi}[1]{%
  \iffieldundef{url}{%
    \iffieldundef{doi}{%
        #1%
      }{% if doi defined
        \href{http://dx.doi.org/\thefield{doi}}{#1}%
    }%
  }{% if url defined
    \href{\thefield{url}}{#1}%
  }%
}
\DeclareFieldFormat[article,misc,incollection,book]{title}{\usebibmacro{string+urldoi}{\mkbibquote{#1}}}

\numberwithin{equation}{section}
\newtheorem{theorem}[equation]{Theorem}

\newtheorem{proposition}[equation]{Proposition}
\newtheorem{corollary}[equation]{Corollary}
\newtheorem{lemma}[equation]{Lemma}

\theoremstyle{definition}

\newtheorem{definition}[equation]{Definition}

\newcommand{\define}[1]{\textit{#1}}

\newcommand{\cstar}{\textup{C}$^*$}

\newcommand{\lmult}{\mathsf{L}}
\newcommand{\rmult}{\mathsf{R}}

\newcommand{\nbar}{|\!|}
\newcommand{\princip}{\delta}
\renewcommand{\epsilon}{\varepsilon}
\newcommand{\intd}{\,\mathrm{d}}
\newcommand{\haar}{\mathrm{m}}

\DeclareMathOperator{\cl}{cl}
\newcommand{\condex}[3][]{\mathbb{E}_{#1}({#2}\vert{#3})}
\newcommand{\N}{\mathbb{N}}

\newcommand{\SL}{\mathrm{SL}}
\newcommand{\PSL}{\mathrm{PSL}}

\renewcommand{\emptyset}{\varnothing}

\newcommand{\alminv}{\mathsf{A}}

\newcommand{\ultra}[1]{\mathsf{{#1}}} % Indicator function
\newcommand{\incl}{\lhook\!\to}
\def\<{\left\langle}
\def\>{\right\rangle}

\DeclareMathOperator{\Cb}{C_{b}}

\DeclareMathOperator{\cont}{C}

\DeclareMathOperator{\conv}{\ast}

\DeclareMathOperator{\dens}{d}

\newcommand{\inv}{\mathcal{I}}

\DeclareMathOperator{\lp}{L}

\DeclareMathOperator{\symdiff}{\triangle}

\DeclareMathOperator{\upperdens}{\overline{\dens}}

\DeclareMathOperator{\extr}{ex}
\newcommand{\cat}[1][top]{\mathcal{C}^{\mathrm{#1}}}

\begin{document}

\title{Triangles in Cartesian Squares of Quasirandom Groups}
\author{Vitaly Bergelson}
\thanks{The first author gratefully acknowledges the support of the NSF under grants DMS-1162073 and DMS-1500575.}
\author{Donald Robertson}
\author{Pavel Zorin-Kranich}
\date{\today}

\begin{abstract}
We prove that triangular configurations are plentiful in large subsets of cartesian squares of finite quasirandom groups from classes having the quasirandom ultraproduct property, for example the class of finite simple groups.
This is deduced from a strong double recurrence theorem for two commuting measure-preserving actions of a minimally almost periodic (not necessarily amenable or locally compact) group on a (not necessarily separable) probability space.
\end{abstract}

\maketitle

\section{Introduction}

By showing that any subset of $\PSL(2,\mathbb{F}_q)$ with density at least $2|\PSL(2,\mathbb{F}_q)|^{8/9}$ contains a subset of the form $\{g,x,gx\}$, Gowers~\cite[Theorem~3.3]{MR2410393} answered negatively a question of Babai and S\'{o}s~\cite{MR810691} on the existence of a constant $c > 0$ such that every finite group $G$ has a subset of size at least $c|G|$ that is \define{product-free}, meaning that it contains no subset of the form $\{g,x,gx\}$.
Gowers also showed~\cite[Lemma~5.1]{MR2410393} that quasirandom groups constitute the general setting for such a result; a finite group is \define{$D$-quasirandom} if and only if it has no non-trivial representations over $\mathbb{C}$ of dimension less than $D$.

The non-existence of large product-free sets in infinite, amenable groups was investigated in \cite{MR2561208}, where it was shown that $G$ is minimally almost periodic (meaning that it has no non-trivial, finite-dimensional, unitary representations over $\mathbb{C}$) if and only if every subset having positive density with respect to some F\o{}lner sequence contains a subset of the form $\{g,x,gx\}$.

In both settings, the non-existence of large product-free sets is related to the absence of finite dimensional representations.
It can also be related to the ergodic theory of group actions.
Indeed, given $A \subset G$, one can find in $A$ a set of the form $\{g,x,gx\}$ if and only if there is $g \in A$ such that $A \cap g^{-1}A \ne \emptyset$ and one is now faced with a question of recurrence for the action of $G$ on itself by left multiplication.
When $G$ is finite, the Haar measure on $G$ is a natural invariant measure, while for countable amenable groups a version of the Furstenberg correspondence principle (e.g.\ \cite[Theorem~4.17]{MR1776759}) can be used to phrase the problem dynamically.
In terms of ergodic theory, then, the question becomes one of relating the representation theory of $G$ to positivity of correlations $\mu(B \cap (T^g)^{-1}B)$ for an action $T$ of $G$ on a probability space $(X,\mathscr{B},\mu)$.

With this framework in mind, the existence of more complicated configurations in subsets of quasirandom groups was considered in \cite{MR3177376}, and in particular the question of whether every large enough subset $A$ of a $D$-quasirandom group contains a configuration of the form $\{g,x,gx,xg\}$ or, equivalently, whether $A \cap g^{-1}A \cap Ag^{-1}$ is non-empty for some $g \in A$.
Dynamically this corresponds to positivity of a multiple correlation of the form
\begin{equation}
\label{eqn:cornerCorrelation}
\mu(B \cap (T_1^g)^{-1} B \cap (T_1^g T_2^g)^{-1}B)
\end{equation}
where $\mu$ is normalized counting measure and $T_1,T_2$ are the commuting actions of $G$ on itself determined by left and right multiplication.
It was shown, in the following strong form, that the conjugation-invariant subsets of $G$ are the only obstruction to the positivity of such correlations.

\begin{theorem}[{\cite[Theorem~5]{MR3177376}}]
\label{thm:bt2Mixing}
Let $G$ be a finite, $D$-quasirandom group with normalized Haar measure $\haar$ and let $f_1,f_2,f_3 : G \to \mathbb{R}$ be bounded in absolute value by 1.
Then
\begin{equation*}
\int \left| \int f_1(x) \cdot f_2(xg) \cdot f_3(gx) \intd\haar(x) - \int f_1 \intd\haar \int f_2 \cdot \condex{f_3}{\mathcal{I}_G} \intd\haar \right| \intd\haar(g) \le c(D)
\end{equation*}
where $\condex{f_3}{\mathcal{I}_G}$ is the orthogonal projection in $\lp^2(G,\haar)$ of $f_3$ on the conjugation-invariant functions and $c(D)$ is a quantity depending only on $D$ that goes to zero as $D \to \infty$.
\end{theorem}

By specializing to $f_1 = f_2 = f_3 = 1_A$ where $A$ is a subset of a $D$-quasirandom group $G$, it follows (see \cite[Corollary~6]{MR3177376}) that for any $\epsilon > 0$ one has
\begin{equation*}
\frac{|A|^3}{|G|^3} - \epsilon \le \frac{|A \cap gA \cap Ag|}{|G|} \le \frac{|A|^2}{|G|^2} + \epsilon
\end{equation*}
for all but at most $\epsilon^{-1}c(D)|G|$ many $g \in G$, where $c(D) \to 0$ as $D \to \infty$.
Thus if $|A| > \epsilon^{-1} c(D) |G|$ then there $g \in A$ for which many configurations of the form $\{x,gx,xg\}$ can be found in $A$.

Theorem~\ref{thm:bt2Mixing} has recently been reproved by Austin~\cite{arxiv:1310.6781} without the use of ultra quasirandom groups, yielding explicit bounds for $c(D)$.
In particular \cite[Theorem~1]{arxiv:1310.6781} implies that for any $D$-quasirandom group $G$ and any $A \subset G$ with $|A|^4 > 4 D^{-1/8} |G|^4$ one has
\[
\haar(A \cap g^{-1}A \cap Ag^{-1}) \ge \haar(A)^3 - \frac{4D^{-1/8}}{\haar(A)} > 0
\]
for some $g \in A$.
See also \cite{arxiv:1405.5629}, where the bound in \cite{arxiv:1310.6781} was improved and Theorem~\ref{thm:bt2Mixing} was generalized to the setting of probability groups.

Correlations of the form \eqref{eqn:cornerCorrelation} control the existence of many other types of configuration.
To describe a couple, in \cite{MR1481813} positivity of \eqref{eqn:cornerCorrelation} on average for arbitrary commuting actions $T_1$ and $T_2$ of any countable, amenable group $G$ on a probability space was proven and used to exhibit triangular configurations of the form $\{ (x,y), (gx,y),(gx,gy) \}$ in any positive-density subset of $G \times G$, and in \cite{arxiv:1402.4736} the correlation \eqref{eqn:cornerCorrelation} was shown to be larger than $\mu(B)^4$ on average when $T_1$ and $T_2$ are any commuting actions of an amenable, minimally almost periodic group having the property that the $G \times G$ action $(g_1,g_2) \mapsto T_1^{g_1} T_2^{g_2}$ is ergodic.
This was used to exhibit two-sided finite products sets in positive-density subsets of such groups.

In this paper we show (see Theorem~\ref{thm:syndetic} below) that there are many triangles, i.e.\ configurations of the form $\{ (x,y), (gx,y),(gx,gy)\}$, in large enough subsets of $G \times G$ provided $G$ is quasirandom enough.
To do this we consider the behavior of the correlation
\begin{equation}
\label{eqn:triangleCorrelation}
\int f_0(x,y) \cdot f_1(gx,y) \cdot f_2(gx,gy) \intd(\haar \times \haar)(x,y)
\end{equation}
for functions $f_0,f_1,f_2 : G \times G \to \mathbb{R}$.
As in \cite{MR3177376} we do not work with a specific quasirandom group $G$ directly, but instead consider the asymptotic behavior of \eqref{eqn:triangleCorrelation} along quasirandom sequences of groups.
\begin{definition}
A sequence $n \mapsto G_n$ of finite, $D_n$-quasirandom groups is a \define{quasirandom} sequence of groups if $D_n \to \infty$ as $n \to \infty$.
\end{definition}
Given a quasirandom sequence $n \mapsto G_n$ of groups, we relate the asymptotic behavior of \eqref{eqn:triangleCorrelation} to a correlation of the form
\begin{equation}
\label{eqn:rothCorrelation}
\int f_0 \cdot T_1^g f_1 \cdot T_1^g T_2^g f_2 \intd\mu
\end{equation}
for commuting actions $T_1$ and $T_2$ of a limiting group $G$ formed from the $G_n$ on a probability space $(X,\mathscr{B},\mu)$.
One can see that if $f_0$ and $f_1$ are supported on disjoint $T_1$-invariant sets then \eqref{eqn:rothCorrelation} is zero, so the correlation depends on the conditional expectations of $f_0$ and $f_1$ on the sub-$\sigma$-algebra of $T_1$-invariant sets.
Similarly, the result depends on the expectations of $f_1$ and $f_2$ on the $T_2$-invariant sets, and on the expectations of $f_0$ and $f_2$ on the $T_1T_2$-invariant sets. 

In order to make precise the dependence of \eqref{eqn:rothCorrelation} on the invariant sub-$\sigma$-algebras mentioned above, one studies the limiting behavior of \eqref{eqn:rothCorrelation} along some limiting scheme, a method that has been in use ever since Furstenberg's ergodic proof~\cite{MR0498471} of Szemer\'{e}di's theorem.
Which limiting scheme is used, and which sub-$\sigma$-algebras control the limiting behavior, depends on the properties of the acting group $G$.

When $G$ is countable and amenable one can use a F\o{}lner sequence $N \mapsto \Phi_N$ to average \eqref{eqn:cornerCorrelation}.
Austin \cite{arxiv:1309.4315} has shown, using his satedness technique -- see Section~\ref{sec:satedness}, that when $(X,\mathscr{B},\mu)$ is a standard probability space, one can find a potentially larger probability space $(Y,\mathscr{D},\nu)$, commuting actions $S_1$ and $S_2$ of $G$ on $(Y,\mathscr{D},\nu)$, and a measurable, measure-preserving map $\pi : Y \to X$ intertwining $T_i$ and $S_i$ such that
\begin{equation}
\label{eqn:cornerFactors}
\int f_0 \cdot T_1^g f_1 \cdot T_1^g T_2^g f_2 \intd\mu -  \int \condex{f_0 \circ \pi}{\alminv_1 \vee \alminv_{12}} \cdot T_1^g \condex{f_1 \circ \pi}{\alminv_1 \vee \alminv_2} \cdot T_1^g T_2^g \condex{f_2 \circ \pi}{\alminv_{12} \vee \alminv_2} \intd\nu
\end{equation}
averaged along any F\o{}lner sequence in $G$ converges to 0, where $\alminv_1,\alminv_2$ and $\alminv_{12}$ are the sub-$\sigma$-algebras of $\mathscr{D}$ generated by the $S_1$, $S_2$ and $S_1 S_2$-invariant functions in $\lp^2(Y,\mathscr{D},\nu)$ respectively.
The $\sigma$-algebras $\alminv_1 \vee \alminv_{12}$, $\alminv_1 \vee \alminv_2$ and $\alminv_{12} \vee \alminv_2$ are called \define{characteristic factors} for the correlation \eqref{eqn:rothCorrelation}.
Austin also gave similar results for longer correlations.

When $G$ is non-amenable, averaging along F\o{}lner sequences is unavailable.
Recently limits along minimal idempotent ultrafilters (idempotents in the Stone--\v{C}ech compactification $\beta G$ of $G$ that belong to a minimal ideal -- see Section~\ref{sec:minimalIdempotents} for details and \cite{MR2354320} for the relative merits of minimal idempotents) have been employed as a replacement.
It was shown in \cite{MR2354320} that for any minimal idempotent ultrafilter $\ultra{p}$ on a countable group $G$ one has
\begin{equation}
\label{eqn:minimalIdempotentFactors}
\lim_{g \to \ultra{p}} \int f_0 \cdot T_1^g f_1 \cdot T_1^g T_2^g f_2 \intd\mu - \int f_0 \cdot T_1^g \condex{f_1}{\mathsf{C}_1} \cdot T_1^g T_2^g \condex{f_2}{\mathsf{C}_{12}} \intd\mu = 0
\end{equation}
where $\mathsf{C}_1$ and $\mathsf{C}_{12}$ are the sub-$\sigma$-algebras corresponding to functions that are almost-periodic for $T_1$ and $T_1T_2$ respectively over $\alminv_2$.
When $G$ is amenable, one can obtain stronger combinatorial results by using minimal idempotent ultrafilters rather than F\o{}lner sequences: this is because positivity of correlations along minimal idempotent ultrafilters yields a larger set of $g \in G$ for which \eqref{eqn:cornerCorrelation} is positive; see \cite{MR2354320} for details.
For limits of longer correlations along minimal idempotent ultrafilters there is no known description of sub-$\sigma$-algebras for which an analogue of \eqref{eqn:minimalIdempotentFactors} holds.
We remark that the difficulty in adapting the techniques in either \cite{arxiv:1309.4315} or \cite{MR2354320} lies in the apparent need to understand certain measures that are not invariant, but merely asymptotically invariant along the ultrafilter.

In this paper we combine Austin's satedness techniques with limits along minimal idempotent ultrafilters to obtain the expected characteristic factors for commuting actions of minimally almost periodic groups.

\begin{theorem}
\label{thm:characteristicFactors}
Let $G$ be a minimally almost periodic group and let $T_1,T_2$ be commuting, measure-preserving actions of $G$ on a compact, Hausdorff probability space $(X,\mu)$ via homeomorphisms.
For any $\epsilon > 0$ and any $f_1$ in $\lp^\infty(X,\mu)$ bounded by $1$ there are commuting, measure-preserving actions $S_1,S_2$ of $G$ on a compact, Hausdorff probability space $(Y,\nu)$ and an intertwining factor map $\pi : Y \to X$ such that
\[
\left| \lim_{g \to \ultra{p}} \int f_0 \cdot T_1^g f_1 \cdot T_1^gT_2^g f_2 \intd\mu - \int \condex{f_0 \circ \pi}{\alminv_1 \vee \alminv_{12}} \cdot S_1^g \condex{f_1 \circ \pi}{\alminv_1 \vee \alminv_2} \cdot S_1^g S_2^g \condex{f_2 \circ \pi}{\alminv_{12} \vee \alminv_2} \intd\nu \right| < \epsilon
\]
for all minimal idempotent ultrafilters $\ultra{p}$ on $G$ and all $f_0,f_2$ in $\lp^\infty(X,\mu)$ bounded by $1$.
\end{theorem}

Note that the space $X$ in Theorem~\ref{thm:characteristicFactors} is not assumed to be metrizable, and that the group $G$ can be uncountable.
For this reason we need to extend Austin's notion of satedness to such spaces; this generalization is carried out in Section~\ref{sec:satedness}.
We then combine Theorem~\ref{thm:characteristicFactors} with an application of Gelfand theory to obtain the following strong recurrence result.

\begin{theorem}
\label{thm:mapCorners}
Let $G$ be a minimally almost periodic group and let $T_1$ and $T_2$ be commuting actions of $G$ on a probability space $(X,\mathscr{B},\mu)$ by measurable, measure-preserving maps.
If $\mu$ is ergodic for the $G \times G$ action $(g_1,g_2) \mapsto T_1^{g_1}T_2^{g_2}$ then
\begin{equation}
\label{eqn:minimalLargeRec}
\lim_{g \to \ultra{p}} \int f_{0} \cdot T_1^g f_{1} \cdot T_1^g T_2^g f_{2} \intd\mu
\ge
\left( \int f_{0}^{1/4}f_{1}^{1/4}f_{2}^{1/4} \intd\mu \right)^4
\end{equation}
for any minimal idempotent ultrafilter $\ultra{p}$ on $G$ and any non-negative measurable functions $f_{0},f_{1},f_{2}$ on $X$.
\end{theorem}

We remark that, by \cite[Theorem~B.1]{arxiv:1309.6095}, the exponent in \eqref{eqn:minimalLargeRec} cannot be improved to $3$ in general.

To deduce the existence of triangles in large enough subsets of quasirandom groups from Theorem~\ref{thm:mapCorners} we need to form a limiting group from a quasirandom sequence $n \mapsto G_n$.
For any sequence $n \mapsto G_n$ of finite groups and any ultrafilter on $\mathbb{N}$ one can form their ultraproduct $G$, which can be given the structure of a probability group using Loeb measure (for details on Loeb measure see, for example \cite{MR720746}).
When the sequence $n \mapsto G_n$ is quasirandom the group $G$ is called an \define{ultra quasirandom group}.
There are commuting actions of $G$ on the ultraproduct $X$ of the sequence $n \mapsto G_n \times G_n$ of groups that correspond to left multiplication by $G_n$ in the first and second coordinates of $G_n \times G_n$ respectively.
In the case that $G$ is minimally almost periodic, we obtain the following result from Theorem~\ref{thm:mapCorners}.

\begin{theorem}
\label{thm:uqrTriangles}
Let $n \mapsto G_n$ be a sequence of finite groups such that their ultraproduct $G$ is minimally almost periodic, and let $\Omega$ be the ultraproduct of the groups $G_n \times G_n$.
Let $L_1$ and $L_2$ be the actions of $G$ on $\Omega$ induced by left multiplication in the first and second coordinates respectively and let $\haar$ be Loeb measure on $\Omega$.
For any minimal idempotent ultrafilter $\ultra{p}$ on $G$ we have
\[
\lim_{g \to \ultra{p}} \int f_{0} \cdot L_1^g f_{1} \cdot L_1^g L_2^g f_{2} \intd\haar
\ge
\left( \int f_{0}^{1/4}f_{1}^{1/4}f_{2}^{1/4} \intd\haar \right)^4
\]
for any non-negative measurable functions $f_{0},f_{1},f_{2}$ on $\Omega$.
\end{theorem}

Theorem~\ref{thm:uqrTriangles} requires that the ultra quasirandom group determined by $n \mapsto G_n$ is minimally almost periodic.
In \cite{MR3177376} it was shown that any ultraproduct of the sequence $n \mapsto \SL(2,\mathbb{F}_{p^n})$ is minimally almost periodic.
More recently, work by Yang~\cite{arxiv:1405.6276} provides many examples of classes $\mathcal{F}$ of groups with the property that the ultraproduct of any quasirandom sequence $n \mapsto G_n$ in $\mathcal{F}$ is minimally almost periodic.
Such classes are called \define{q.u.p.\ (quasirandom ultra product)} classes.
For example, the class of finite, quasisimple groups is q.u.p.\ by \cite[Corollary~1.12]{arxiv:1405.6276}.

\begin{theorem}
\label{thm:syndetic}
Let $\mathcal{F}$ be a q.u.p.\ class of finite groups.
For every $0<\alpha<1$ and every $\epsilon>0$ there exist $D,K\in\N$ such that for every $D$-quasirandom group $G\in\mathcal{F}$ and every $A\subset G\times G$ with $|A|\geq\alpha |G|^{2}$ the set
\begin{equation}
\label{eqn:triangleReturns}
\left\{ g \in G : \frac{|A\cap (1,g)^{-1}A \cap (g,g)^{-1}A|}{|G|^2} > \alpha^{4}-\epsilon\right\}
\end{equation}
has the property that at most $K$ of its right shifts are needed to cover $G$.
\end{theorem}

It would be interesting to obtain a version Theorem~\ref{thm:syndetic} with explicit description of $D$ and $K$.
Such a proof may also shed light on the question of how large the set \eqref{eqn:triangleReturns} can be: we conjecture that its density in $G$ should tend to 1 as $D \to \infty$ in analogy with Theorem~\ref{thm:bt2Mixing}, but have been unable to prove this using our techniques.%
\footnote{These problems have been solved by  Austin \cite{arxiv:1503.08746} after the completion of this article.}
Combined with the fact that the non-cyclic finite simple groups are quasirandom in the sense that the minimal dimension of a non-trivial irreducible representation grows with the order of the group (see \cite[Theorem 4.7]{MR2410393}), Theorem~\ref{thm:syndetic} yields the following consequence:
\begin{corollary}
\label{cor:syndetic}
For every $0<\alpha<1$ and every $\epsilon>0$ there exist $N,K\in\N$ such that for every non-cyclic finite simple group $G$ of order at least $N$ and every $A\subset G\times G$ with $|A|\geq\alpha |G|^{2}$ the set \eqref{eqn:triangleReturns} has the property that at most $K$ of its right shifts are needed to cover $G$.
\end{corollary}

We conclude by mentioning that an alternative method for forming a limiting group from a quasirandom sequence is available when the sequence is increasing.
Combinatorially, it gives many $g \in A$ for which $A \cap (1,g)^{-1} A \cap (g,g)^{-1}A \ne \emptyset$, but does not give much information about the size of the intersection.
Given a sequence $n \mapsto G_n$ of groups such that $G_n \incl G_{n+1}$ for all $n \in \mathbb{N}$, the direct limit $G$ is the union of the embeddings $G_n \incl G$ (see \cite[Page~23]{MR1357169}) and is therefore amenable.
If, in addition, the sequence is quasirandom, then $G$ is minimally almost periodic.
%There are many examples of amenable, minimally almost periodic groups: see \cite{MR3071509} and \cite{arxiv:1405.7605}.
For such sequences we apply \cite[Corollary~4.9]{arxiv:1402.4736} to obtain the following combinatorial result.

\begin{theorem}
\label{thm:increasingGroups}
Let $n \mapsto G_n$ be a quasirandom sequence such that $G_n \incl G_{n+1}$ for all $n \in \mathbb{N}$.
For every $\alpha > 0$ and every $\epsilon > 0$ there is $N \in \mathbb{N}$ such that, for any $n \ge N$ and any $A \subset G_n \times G_n$ with $|A| \ge \alpha |G_n|^2$ there are $(1 - \epsilon)|G_n|$ many $g \in G_n$ for which $A \cap (1,g)^{-1} A \cap (g,g)^{-1}A$ is non-empty.
\end{theorem}

The rest of the paper runs as follows.
In Section~\ref{sec:systems} we define the  categories of dynamical systems we will work with and recall the Jacobs--de Leeuw--Glicksberg decomposition.
A version of satedness suitable for our needs is developed in Section~\ref{sec:satedness}.
In Sections~\ref{sec:minimalIdempotents} and ~\ref{sec:mapGroups} we present the necessary facts regarding minimal idempotent ultrafilters and minimally almost periodic groups respectively.
Theorem~\ref{thm:characteristicFactors} is proved in Section~\ref{sec:chf} and Theorem~\ref{thm:mapCorners} is proved in Section~\ref{sec:strongRec}.
Lastly, the combinatorial results mentioned above are proved in Section~\ref{sec:combinatorics}.

We would like to thank the referees for constructive and thorough reports.

\section{Topological and Measure-preserving dynamical systems}
\label{sec:systems}

In this section various categories of dynamical systems are defined that will be used throughout the paper.

\begin{definition}
Let $G$ be a discrete group.
The objects of the category $\cat(G)$ are the left actions $T : G \times X \to X$ of $G$ on compact, Hausdorff spaces $X$ having the property that each of the induced maps $T^g : X \to X$ is continuous.
Objects in $\cat(G)$ will be called \define{systems} and denoted $\mathbf{X}=(X,T)$.
Their defining continuous left actions will be called just \define{actions}.
The morphisms of the category $\cat$ are the continuous maps intertwining the $G$ actions on the domain and the codomain.
\end{definition}

Note that the objects of $\cat(G)$ may be actions on non-metrizable topological spaces.
This level of generality is needed in order to handle actions of very large groups on the Gelfand spaces of non-separable \cstar{}-algebras, which will play a role in the next section.

In our applications the acting group will be $G^{2}$ and an action in $\cat(G^2)$ will be written in the form $(g_1,g_2) \mapsto T_{1}^{g_1} T_{2}^{g_2}$ where $T_{1}$ and $T_{2}$ are commuting $G$ actions.
We also write $T_{12}^{g} = T_{1}^{g}T_{2}^{g}$.

For any system $\mathbf{X}$ we denote by $M_{\mathbf{X}}$ the set of Baire probability measures on $X$ that are $T$-invariant.
(Recall that the \define{Baire} sets are the members of the $\sigma$-algebra generated by the compact $G_{\delta}$ sets.)
In view of the Riesz--Markov--Kakutani representation theorem this set can be seen as a subset of the dual of the space $\cont(X)$ of continuous, real-valued functions on $X$ equipped with the uniform norm, and as such it is compact and convex with respect to the weak$^{*}$ topology.
The set of extreme points of $M_{\mathbf{X}}$ is denoted $\extr M_{\mathbf{X}}$.
It follows from the Radon--Nikodym theorem that the extreme points of $M_{\mathbf{X}}$ are precisely the ergodic measures, namely the measures for which every almost invariant Baire set has measure either $0$ or $1$.

\begin{definition}
The objects of the category $\cat[meas]$ are pairs $(\mathbf{X},\mu)$, where $\mu\in M_{\mathbf{X}}$, called \define{measure-preserving systems}.
A morphism $(\mathbf{Y},\nu) \to (\mathbf{X},\mu)$ in $\cat[meas]$ is any morphism $\pi : \mathbf{Y} \to \mathbf{X}$ in $\cat$ such that $\pi \nu = \mu$.
When there is a morphism $\pi : (\mathbf{Y},\nu) \to (\mathbf{X},\mu)$ in $\cat[meas]$ we call $\pi$ the \define{factor map} and say that $(\mathbf{Y},\nu)$ is an \define{extension} of $(\mathbf{X},\mu)$, or that $(\mathbf{X},\mu)$ is a \define{factor} of $(\mathbf{Y},\nu)$.
The category $\cat[erg]$ is the subcategory of $\cat[meas]$ whose objects are the pairs $(\mathbf{X},\mu)$ for which $\mu\in \extr M_{\mathbf{X}}$.
The objects of $\cat[erg]$ are called \define{ergodic systems}.
\end{definition}

\begin{lemma}
\label{lem:extremeLiftsErgodic}
Let $\psi : \mathbf{Y} \to \mathbf{X}$ be a morphism in $\cat$ and fix $\mu$ in $\extr M_{\mathbf{X}}$.
Then $\tilde{M}_\mathbf{Y} = \{ \nu \in M_\mathbf{Y} : \psi \nu = \mu \}$ is a compact, convex set and its extreme points are ergodic measures on $\mathbf{Y}$.
\end{lemma}
\begin{proof}
It is clearly closed and convex.
To see that its extreme points are ergodic, it suffices to show that $\extr \tilde M_{\mathbf{Y}} \subset \extr M_{\mathbf{Y}}$.
Indeed, suppose $\lambda\in \extr \tilde{M}_{\mathbf{Y}}$ can be written as $\lambda = c\lambda_{1} + (1-c)\lambda_{2}$ with $\lambda_{1},\lambda_{2}\in M_{\mathbf{Y}}$.
Then $c \pi_{*} \lambda_{1} + (1-c) \pi_{*} \lambda_{2} = \mu$, so by extremality of $\mu$ in $M_{\mathbf{X}}$ we have $\pi_{*}\lambda_{1}=\pi_{*}\lambda_{2}=\mu$.
By extremality of $\lambda$ in $\tilde{M}_{\mathbf{Y}}$ this implies $\lambda_{1}=\lambda_{2}=\lambda$.
\end{proof}

We will be concerned with sub-$\sigma$-algebras of invariant sets.
Given a measure space $(X,\mathscr{B},\mu)$ and sets $A,B \in \mathscr{B}$, write $A \sim B$ when $\mu(A \symdiff B) = 0$.

\begin{definition}
Let $(\mathbf{X},\mu)$ be a measure-preserving system in $\cat[meas](G)$.
A Baire subset $B \subset X$ is \define{almost invariant with respect to $\mu$} if $(T^g)^{-1} B \sim B$ for every $g \in G$.
Write $\alminv^\mu \mathbf{X}$ for the sub-$\sigma$-algebra generated by the almost invariant sets.
If $(\mathbf{X},\mu)$ is a measure-preserving system in $\cat[meas](G^2)$ and $i \in \{1,2,12\}$, write $\alminv_i^\mu \mathbf{X}$ for the sub-$\sigma$-algebra generated by the $T_i$ almost invariant sets.
Lastly, for any $i,j \in \{1,2,12\}$ define $\alminv_{i,j}^\mu \mathbf{X} = \alminv_i^\mu \mathbf{X} \vee \alminv_j^\mu \mathbf{X}$.
\end{definition}

We conclude this section by recalling a general version of the splitting of $\lp^2(\mathbf{X},\mu)$ into almost periodic and weakly mixing parts that will be needed in the proof of Theorem~\ref{thm:mapCorners}.
Let $G$ be a group and let $(\mathbf{X},\mu)$ be a measure preserving system in $\cat[meas](G)$.
Write $\mathscr{H}$ for $\lp^2(\mathbf{X},\mu)$.
Consider the collection $\mathscr{S} = \{T^g : g \in G \}$ of unitary operators on $\mathscr{H}$.
The closure with respect to the weak topology of any orbit of $\mathscr{S}$ is compact in the weak topology, so $\mathscr{S}$, in the terminology of \cite{MR0131784}, is a weakly almost periodic semigroup of operators.
Write $\overline{\mathscr{S}}$ for the closure of $\mathscr{S}$ in the weak operator topology.
Applying \cite[Corollary~4.12]{MR0131784} allows us to write $\mathscr{H} = \mathscr{H}_\mathsf{o} + \mathscr{H}_\mathsf{r}$ where
\begin{equation*}
\mathscr{H}_\mathsf{o} = \{ f \in \mathscr{H} : 0 \in \overline{\mathscr{S}} f \}
\end{equation*}
is the closed subspace of \define{flight vectors} and
\begin{equation*}
\mathscr{H}_\mathsf{r} = \{ f \in \mathscr{H} : \overline{\mathscr{S}} f = \overline{\mathscr{S}} h \textrm{ for every } h \in \overline{\mathscr{S}} f \}
\end{equation*}
is the closed subspace of \define{reversible vectors}.
Since $\mathscr{S}$ is a group \cite[Lemma~4.5]{MR0131784} implies that $\mathscr{H}_\mathsf{r}$ is spanned by the finite dimensional, $\mathscr{S}$-invariant subspaces of $\mathscr{H}$.
The splitting $\mathscr{H}_\mathsf{o} + \mathscr{H}_\mathsf{r}$ is determined by the unique projection in the kernel of $\overline{\mathscr{S}}$.
By \cite[Theorem~2.3(iv)]{MR0131784} this kernel is self-adjoint so the unique projection in the kernel is also self-adjoint, implying that the above splitting is orthogonal.

Let $(\mathbf{X}_1,\mu_1)$ and $(\mathbf{X}_2,\mu_2)$ be measure-preserving systems in $\cat[meas](G)$ and let $\mathscr{H}_1$ and $\mathscr{H}_2$ be the corresponding Hilbert spaces.
The Hilbert space corresponding to the product system $(\mathbf{X}_1 \times \mathbf{X}_2, \mu_1 \times \mu_2)$ is $\mathscr{H}_1 \otimes \mathscr{H}_2$.
We will need the following result, which relates the splittings of these Hilbert spaces, in the proof of Theorem~\ref{thm:mapCorners}.

\begin{lemma}
\label{lem:jdgSplitting}
Let $(\mathbf{X}_1,\mu_1)$ and $(\mathbf{X}_2,\mu_2)$ be measure-preserving systems in $\cat[meas](G)$ and let $\mathscr{H}_1$ and $\mathscr{H}_2$ be the associated $\lp^2$ spaces.
Let $\mathscr{H}$ be the $\lp^2$ space of the product measure-preserving system $(\mathbf{X}_1 \times \mathbf{X}_2, \mu_1 \otimes \mu_2)$.
Then $\mathscr{H}_{1,\mathsf{r}} \otimes \mathscr{H}_{2,\mathsf{r}} = \mathscr{H}_\mathsf{r}$.
\end{lemma}
\begin{proof}
Write
\[
\mathscr{H}_{\mathsf{o}} \oplus \mathscr{H}_{\mathsf{r}}
=
\mathscr{H}
=
(\mathscr{H}_{1,\mathsf{r}} \otimes \mathscr{H}_{2,\mathsf{r}})
\oplus
(\mathscr{H}_{1,\mathsf{o}} \otimes \mathscr{H}_{2,\mathsf{r}})
\oplus
(\mathscr{H}_{1,\mathsf{r}} \otimes \mathscr{H}_{2,\mathsf{o}})
\oplus
(\mathscr{H}_{1,\mathsf{o}} \otimes \mathscr{H}_{2,\mathsf{o}})
\]
and note that $\mathscr{H}_{1,\mathsf{r}} \otimes \mathscr{H}_{2,\mathsf{r}} \subset \mathscr{H}_\mathsf{r}$, whereas
 if $f \in \mathscr{H}_{1,\mathsf{o}}$ and $g \in \mathscr{H}_2$, or if $f \in \mathscr{H}_1$ and $g \in \mathscr{H}_{2,\mathsf{o}}$, then $f \otimes g \in \mathscr{H}_\mathsf{o}$.
\end{proof}

\section{Satedness}
\label{sec:satedness}

Satedness, introduced by Austin~\cite{MR2599882,arxiv:1309.4315} to prove convergence of multiple ergodic averages, is a property that a measure-preserving system may possess with respect to certain classes of systems.
For example, one could speak of satedness with respect to the class of Kronecker systems.
Although the concept depends critically on an invariant measure, it will be convenient to consider classes defined by topological, rather than measure-theoretic properties.
We will therefore consider satedness with respect to idempotent classes, defined below, in $\cat$.
Recall that a \define{joining} of objects $\mathbf{X}_1,\dots,\mathbf{X}_n$ in $\cat$ is an object $\mathbf{Z}$ in $\cat$ together with factor maps $\mathbf{Z} \to \mathbf{X}_i$ for each $1 \le i \le n$.

\begin{definition}
An \define{idempotent class} in $\cat$ is a class that contains the one-point system, is closed under joinings of finitely many systems, and is closed under arbitrary inverse limits.
\end{definition}

It follows from Zorn's lemma that every idempotent class $\mathcal{I}$ defines a map $\mathcal{I} : \mathbf{X} \mapsto \mathcal{I}\mathbf{X}$ on $\cat$ that associates to every system its maximal factor in $\mathcal{I}$.
We remark that this map can be seen as a natural transformation from the identity functor on $\cat$.

The following idempotent classes on $\cat(G^2)$ will play a crucial role in what follows.

\begin{definition}
Let $G$ be a group and let $\mathbf{X}$ be a system in $\cat(G^2)$.
For $i=1,2,12$ let $\inv_{i}$ be the idempotent class of systems on which the action $T_{i}$ is trivial.
Define $\inv_{i,j}=\inv_{i}\vee\inv_{j}$ for any $i,j \in \{1,2,12\}$.
\end{definition}

Note that, for any $\mu \in M_\mathbf{X}$ the space $\lp^{2}(\inv_{1,2}\mathbf{X},\mu)$ is a priori smaller than $\lp^{2}(\alminv^\mu_{1,2}\mathbf{X},\mu)$.

We now turn to our version of satedness, which is based on Austin's definition in \cite{arxiv:1309.4315}, but differs in that the energy increment \eqref{eqn:satedness} below is quantified, rather than being required to vanish.

\begin{definition}
\label{def:sated}
Let $\cat[]$ be a subcategory of $\cat[meas]$ and let $\mathcal{I}$ be an idempotent class (or, more generally, a natural transformation from the identity functor in $\cat$).
A measure-preserving system $(\mathbf{X},\mu)$ is called \define{$(\epsilon,f,\mathcal{I})$ sated in $\mathcal{C}$} for $\epsilon \geq 0$ and $f \in \lp^2(\mathbf{X},\mu)$ if one of the following equivalent conditions holds.
\begin{enumerate}
\item For every extension $\pi : (\mathbf{Y},\nu) \to (\mathbf{X},\mu)$ in $\cat[]$ we have
\begin{equation}
\label{eqn:satedness}
\nbar \condex[\nu]{f \circ \pi}{\mathcal{I}\mathbf{Y}} - \condex[\mu]{f}{\mathcal{I}\mathbf{X}} \circ \pi \nbar_{2} \leq \epsilon.
\end{equation}
\item For every extension $\pi : (\mathbf{Y},\nu) \to (\mathbf{X},\mu)$ in $\cat[]$ and every $\phi \in \cont(\mathcal{I}\mathbf{Y})$ we have
\begin{equation}
\label{eqn:satedEquiv}
\left| \int f \circ \pi \cdot \phi \intd\nu - \int \condex[\mu]{f}{\mathcal{I}\mathbf{X}} \circ \pi \cdot \phi \intd\nu \right| \leq \epsilon \nbar \phi \nbar_{2}.
\end{equation}
\end{enumerate}
\end{definition}

The conditions \eqref{eqn:satedness} and \eqref{eqn:satedEquiv} are equivalent because $\nbar v \nbar = \sup \left\{ |\langle v,w \rangle| : \nbar w \nbar = 1 \right\}$ in any Hilbert space and $\cont(\mathcal{I}\mathbf{Y})$ is dense in $\lp^2(\mathcal{I} \mathbf{Y},\mu)$.

The following result shows that every measure-preserving system has an extension that is sated up to any prescribed error.
We will be able to do this without the use of inverse limits because the energy increment in \eqref{eqn:satedness} is only required to be small rather than to vanish.
This extends the applicability of satedness to categories where inverse limits may not exist.

\begin{theorem}
\label{thm:epsilon-sated}
Let $\cat[]$ be a subcategory of $\cat[meas]$ and let $\mathcal{I}$ be an idempotent class in $\cat$, or more generally a natural transformation from the identity functor in $\cat$.
Let $(\mathbf{X},\mu)$ be a measure-preserving system in the category $\cat[]$.
Then for any $\epsilon > 0$ and any $f \in \lp^2(\mathbf{X},\mu)$ there exists an extension $\psi : (\mathbf{Y},\nu) \to (\mathbf{X},\mu)$ of measure-preserving systems such that $(\mathbf{Y},\nu)$ is $(\epsilon,f\circ\psi,\mathcal{I})$ sated in $\cat[]$.
\end{theorem}
\begin{proof}
Fix $\epsilon > 0$ and $f \in \lp^2(\mathbf{X},\mu)$.
Assume $f \ne 0$ as otherwise the conclusion is immediate.
We have
\[
\sup \{ \nbar \condex{f\circ\psi}{\mathcal{I}\mathbf{Y}} \nbar_{2} : (\mathbf{Y},\nu) \stackrel{\psi}{\to} (\mathbf{X},\mu) \textrm{ a morphism} \} \leq \nbar f\nbar_{2} < \infty,
\]
so there exists an extension $\psi : (\mathbf{Y},\nu) \to (\mathbf{X},\mu)$ such that $\nbar \condex{f \circ \psi}{\mathcal{I}\mathbf{Y}} \nbar$ is within $\epsilon/(2 \nbar f \nbar_2)$ of the supremum above.
For any further extension $\pi : (\mathbf{Z},\lambda) \to (\mathbf{Y},\nu)$ we have
\[
\condex{f\circ\psi}{\mathcal{I}\mathbf{Y}} \circ\pi
=
\condex{f\circ\psi\circ\pi}{\pi^{-1}(\mathcal{I}\mathbf{Y})}
=
\condex{\condex{f\circ\psi\circ\pi}{\mathcal{I}\mathbf{Z}}}{\pi^{-1}(\mathcal{I}\mathbf{Y})}
\]
by functoriality of $\mathcal{I}$.
On the other hand, by choice of $\mathbf{Y}$ we have
\[
\nbar  \condex{f\circ\psi}{\mathcal{I}\mathbf{Y}} \nbar_{2} > \nbar \condex{f\circ\psi\circ\pi}{\mathcal{I}\mathbf{Z}} \nbar_{2} - \delta.
\]
Since the conditional expectation onto $\pi^{-1}(\mathcal{I}\mathbf{Y})$ is an orthogonal projection, this implies
\[
\nbar \condex{f\circ\psi}{\mathcal{I}\mathbf{Y}} - \condex{f\circ\psi\circ\pi}{\mathcal{I}\mathbf{Z}} \nbar_{2}
\leq 2 \nbar\condex{f\circ\psi\circ\pi}{\mathcal{I}\mathbf{Z}}\nbar_{2} \delta
\leq 2 \nbar f\nbar_{2} \delta,
\]
so $(\mathbf{Y},\nu)$ is $(\epsilon,f\circ\psi,\mathcal{I})$ sated provided $\delta<\epsilon/(2\nbar f\nbar_{2})$.
\end{proof}

Theorem~\ref{thm:epsilon-sated} will be applied to the category $\cat[erg]$ of ergodic systems.
However, we will need satedness in the class of \emph{all} measure-preserving systems.
Switching between these classes requires a version of the ergodic decomposition.
Since we do not assume metrizability of the compact spaces under consideration, we use the following Choquet-type theorem.
Recall that a function $f$ from a convex set $M$ to $\mathbb{R}$ is \define{affine} if $f(tx + (1-t)y) = tf(x) + (1-t)f(y)$ for all $0 \le t \le 1$ and all $x,y \in M$.

\begin{theorem}[{Choquet--Bishop--de Leeuw, \cite[p.\ 17]{MR1835574}}]
\label{thm:choquet-bishop-de_leeuw}
Suppose that $M$ is a compact convex subset of a locally convex space and let $\mu\in M$.
Then there exists a probability measure $\eta$ on $M$ that represents $\mu$ in the sense that
\[
\phi(\mu) = \int \phi \intd \eta
\]
for every continuous affine function $\phi : M \to \mathbb{R}$ and such that $\eta$ vanishes on every Baire subset of $M$ that is disjoint from $\extr M$.
\end{theorem}
In this version of the Choquet theorem the representing measure $\eta$ is not unique and is only supported by the extreme points in a weak sense, but this will not be an issue.

\begin{lemma}
\label{lem:top-extension}
Let $(\mathbf{X},\mu)$ be a measure-preserving system and $F\subset \lp^{\infty}(X)$.
Then there exists an extension $\pi:(\mathbf{Y},\nu)\to (\mathbf{X},\mu)$ such that $f\circ\pi$ coincides with a continuous function on $Y$ $\nu$-a.e.\ for every $f\in F$.
\end{lemma}
\begin{proof}
Let $A$ be the minimal $G$-invariant \cstar-subalgebra of $\lp^{\infty}(X,\mu)$ that contains $\cont(X) \cup F$.
Let $Y$ be its Gelfand spectrum with the canonical $G$-action and the canonical projection $\pi$ onto $X$.
We have a positive linear functional $\nu$ on $A$ given by $\nu(g)=\int g \intd\mu$, this defines a $G$-invariant probability measure on $Y$.

It remains to show that $f\circ\pi$ coincides with a continuous function $\nu$-a.e.\ for every $f\in F$.
Fix $f\in F$, by duality it suffices to verify
\begin{equation}
\label{eq:two-lin-forms}
\int (f\circ\pi) g \intd\nu = \int \tilde f g \intd\nu
\end{equation}
for every $g\in L^{1}(Y,\nu)$, where $\tilde f$ is the continuous function on $Y$ corresponding to $f$ viewed as an element of $A$.
Both $f\circ\pi$ and $\tilde f$ are bounded functions, so it suffices to verify this identity for $g$ in a dense subspace of $L^{1}(Y,\nu)$.
We claim that $\pi^{*}\cont(X)$ is one such subspace.
Indeed, $\cont(Y)$ is dense in $L^{1}(Y,\nu)$ and $\cont(X)$ is $L^{1}$-dense in $A$.
%Since absolute value is a \cstar-algebraic operation, t
This implies that $\pi^{*}\cont(X)$ is $L^{1}$ dense in $\cont(Y)$.
For every $g\in \cont(X)$ we have $\tilde g=g\circ\pi$, so \eqref{eq:two-lin-forms} boils down to
\[
\int ((fg)\circ\pi) \intd\nu = \int fg \intd\mu
\]
as desired.
\end{proof}

\begin{proposition}
\label{prop:erg-sated}
Let $(\mathbf{X},\mu)$ be $(\epsilon,f,\mathcal{I})$-sated in $\cat[erg]$ for some $f\in\lp^{\infty}(X,\mu)$ and $\epsilon\geq 0$.
Then $(\mathbf{X},\mu)$ is also $(\epsilon,f,\mathcal{I})$-sated in $\cat[meas]$.
\end{proposition}
\begin{proof}
Let $\pi : (\mathbf{Y},\nu) \to (\mathbf{X},\mu)$ be an extension in $\cat[meas]$.
We have to show \eqref{eqn:satedEquiv} for every $\phi\in\cont(\mathcal{I}\mathbf{Y})$.
Passing to a further extension of $(\mathbf{Y},\nu)$ using Lemma~\ref{lem:top-extension}, we may assume that both $f$ and $\tilde f = \condex[\mu]{f}{\mathcal{I}X} \circ\pi$ (which are a priori merely bounded measurable functions) admit representatives in $\cont(Y)$.
While verifying
\[
\left| \int f \phi \intd \nu - \int \tilde f \phi \intd \nu \right| \leq \epsilon \nbar \phi\nbar_{\lp^{2}(\nu)}
\]
for every $\phi\in\cont(\mathcal{I}\mathbf{Y})$ there is no harm in replacing $f$ and $\tilde f$ by their continuous representatives.

Write $\tilde{M}_\mathbf{Y}$ for the measures in $M_\mathbf{Y}$ that extend $\mu$.
By Lemma~\ref{lem:extremeLiftsErgodic} it is a closed, convex subset of $M_\mathbf{Y}$ whose extreme points are ergodic.
Thus for any measure $\lambda \in \extr \tilde{M}_{\mathbf{Y}}$ we have
\begin{equation}
\label{eq:satedness-extreme}
\left| \int f \phi \intd \lambda - \int \tilde{f} \phi \intd \lambda \right| \leq \epsilon \nbar \phi\nbar_{\lp^{2}(\lambda)}.
\end{equation}
by the satedness hypothesis.
Let now $\eta$ be a measure on $\tilde{M}_{\mathbf{Y}}$ representing $\nu$ in the sense of the Choquet--Bishop--de Leeuw theorem (Theorem \ref{thm:choquet-bishop-de_leeuw}).
Consider the set
\[
\Lambda = \left\{ \lambda\in \tilde{M}_{\mathbf{Y}} : \left| \int f\phi \intd \lambda - \int \tilde{f} \phi \intd \lambda\right| > \epsilon\nbar\phi\nbar_{\lp^{2}(\lambda)}\right\}
\]
which is disjoint from $\extr \tilde{M}_{\mathbf{Y}}$ in view of \eqref{eq:satedness-extreme}, and Baire because it consists of those measures $\lambda$ where one continuous function of $\lambda$ is larger than another.
%\[
%\Lambda = \bigcup_{\epsilon'\in\Q, \epsilon'>\epsilon} \bigcap_{\delta\in\Q, \delta>0} \{ \lambda\in \tilde M_{Y} : \big| \int fg \intd \lambda - \int \tilde f g \intd \lambda\big| > (\epsilon'-\delta) (\int |g|^{2} \intd\lambda )^{1/2} \}.
%\]
It follows that $\eta(\Lambda) = 0$.
Therefore
\begin{align*}
\left| \int f \phi \intd \nu - \int \tilde{f} \phi \intd \nu \right|
&
\leq
\int\limits_{\tilde M_{\mathbf{Y}}} \left| \int f \phi \intd \lambda - \int \tilde{f} \phi \intd \lambda \right| \intd\eta(\lambda)
\\
&
=
\int\limits_{\tilde M_{\mathbf{Y}}\setminus\Lambda} \left| \int f \phi \intd \lambda - \int \tilde{f} \phi \intd \lambda \right| \intd\eta(\lambda)
\leq
\int\limits_{\tilde M_{\mathbf{Y}}\setminus\Lambda} \epsilon\nbar\phi\nbar_{L^{2}(\lambda)} \intd\eta(\lambda)
\leq
\epsilon \nbar\phi\nbar_{\lp^{2}(\nu)}
\end{align*}
by H\"older's inequality.
\end{proof}

\section{Minimal idempotent ultrafilters}
\label{sec:minimalIdempotents}

For any non-empty set $X$ write $\beta X$ for the collection of ultrafilters on $X$.
Recall that these are the filters on $X$ that are maximal with respect to containment, and can be thought of as finitely-additive $\{0,1\}$-valued measures on $X$.
We identify each $x \in X$ with the principal ultrafilter $\princip_x = \{ A \subset X : x \in A \}$.
Upon equipping $\beta X$ with the topology defined by the base consisting of the clopen sets $\overline{A} = \{ \ultra{p} \in \beta X : A \in \ultra{p} \}$ for any subset $A$ of $X$ it becomes a compact, Hausdorff topological space.
It enjoys the following universal property, which will be used repeatedly as a means to take limits along ultrafilters.

\begin{proposition}
\label{prop:ultrafilterExtension}
Let $X$ be a non-empty set.
For any compact, Hausdorff topological space $Z$ and any map $\phi : X \to Z$ there is a continuous map $\beta X \to Z$ that agrees with $\phi$ on the principal ultrafilters.
\end{proposition}

Given a map $\phi$ from $X$ to a compact, Hausdorff space, we denote by
\begin{equation*}
\lim_{x \to \ultra{p}} \phi(x)
\end{equation*}
the value at $\ultra{p} \in \beta X$ of the extension provided by Proposition~\ref{prop:ultrafilterExtension}.

Let $G$ be any group.
One can make $\beta G$ a semigroup by defining
\begin{equation}
\label{eqn:ultrafilterConvolution}
\ultra{p} \conv \ultra{q} = \{ A \subset G : \{ g \in G :A g^{-1} \in \ultra{p} \} \in \ultra{q} \}
\end{equation}
for any $\ultra{p},\ultra{q} \in \beta G$.
Note that $\delta_g \conv \delta_h = \delta_{gh}$ for any $g,h \in G$ so \eqref{eqn:ultrafilterConvolution} extends multiplication on $G$.
The operation above makes $\beta G$ a right semi-topological semigroup: for any fixed $\ultra{p}$ in $\beta G$ the map $\ultra{q} \mapsto \ultra{p} \conv \ultra{q}$ is continuous.
Ellis's lemma \cite[Lemma~1]{MR0101283} implies that there are idempotents for \eqref{eqn:ultrafilterConvolution} in any compact sub-semigroup of $\beta G$.
The semigroup operation on $\beta G$ interacts with continuous actions of $G$ on compact, Hausdorff spaces in the following way (cf. \cite[Lemma~6.1]{MR982232}).

\begin{proposition}
\label{prop:ultrafilterAction}
Let $G$ be a group and let $T : G \times X \to X$ be a right actions of $G$ on a compact, Hausdorff space $X$ via continuous maps.
Then
\begin{equation*}
\lim_{g \to \ultra{p} \conv \ultra{q}} T^g x = \lim_{g \to \ultra{q}} \lim_{h \to \ultra{p}} T^g (T^h x)
\end{equation*}
for all $x \in X$ and all $\ultra{p},\ultra{q} \in \beta G$.
In particular, if $\ultra{p}$ is idempotent then
\[
\lim_{g \to \ultra{p}} T^g x = \lim_{g \to \ultra{p}} \lim_{h \to \ultra{p}} T^g(T^h x)
\]
for all $x \in X$.
\end{proposition}

The following version of the van der Corput trick follows immediately from the proof of \cite[Lemma~4]{arxiv:0711.0484}.

\begin{proposition}
\label{prop:vdc}
Let $G$ be a group and let $\mathscr{H}$ be a Hilbert space.
For any sequence $u : G \to \mathscr{H}$ that is norm-bounded and any idempotent ultrafilter $\ultra{p}$ on $G$, if
\begin{equation*}
\left| \lim_{h \to \ultra{p}} \lim_{g \to \ultra{p}} \langle u(hg),u(g) \rangle \right| < \epsilon
\end{equation*}
then $\nbar \lim\limits_{g \to \ultra{p}} u(g) \nbar^2 \le \epsilon$, the latter limit being taken in the weak topology on $\mathscr{H}$.
\end{proposition}

An idempotent ultrafilter $\ultra{p}$ is \define{minimal} if it belongs to a minimal right ideal in $\beta G$.
Every non-zero right ideal contains a minimal right ideal, so Ellis's lemma \cite[Lemma~1]{MR0101283} implies that every right ideal contains a minimal idempotent ultrafilter.
See e.g.\ \cite[Section~2]{MR2354320} for the details.
The following lemma tells us that sets in minimal idempotent ultrafilters have a certain largeness property.
Recall that $S \subset G$ is \define{right syndetic} if there is a finite subset $F$ of $G$ for which $S F = G$.

\begin{lemma}[cf.~{\cite[Theorem~2.4]{MR2052273}}]
\label{lem:centralSyndeticity}
Let $G$ be a group and let $\ultra{p}$ be a minimal idempotent ultrafilter on $G$.
For any $A \in \ultra{p}$ the set $A^{-1}A = \{ g^{-1}h : g,h \in A \}$ is right syndetic.
\end{lemma}
\begin{proof}
Fix $A \in \ultra{p}$.
Let $X$ be a minimal right ideal containing $\ultra{p}$.
Since $\ultra{p} \conv \beta G$ is a right ideal contained in $X$ it must be equal to $X$, so continuity of the map $\ultra{q} \mapsto \ultra{p} \conv \ultra{q}$ implies $X$ is compact.
Consider the continuous right action $T$ of $G$ on $X$ defined by $T^g(\ultra{p}) = \ultra{p} \conv \delta_g$.
The set $U := \overline{A} \cap X$ is open in $X$ and contains $\ultra{p}$.
We claim that the collection $\{ (T^g)^{-1} U : g \in G\}$ covers $X$.
Indeed, if not then the complement $V$ of its union is a closed, non-empty, $T$-invariant subset of $X$.
This implies that $V$ is a right ideal, because any $\ultra{q}$ in $V$ satisfies $\ultra{q} \conv \beta G = \ultra{q} \conv \overline{G} = \cl(\ultra{q} \conv G) \subset \cl(V) = V$ where $\cl$ denotes the closure of a set in $\beta G$.

Since the sets $(T^g)^{-1} U$ cover $X$ we can extract a finite subcover $(T^{g_1})^{-1} U,\dots,(T^{g_n})^{-1} U$.
Thus for every $g \in G$ there is some $1 \le i \le n$ for which $T^g(\ultra{p}) \in (T^{g_i})^{-1} U$.
We can rewrite this as $\ultra{p} \conv \princip_{gg_i} \in U$, which is the same as $A(gg_i)^{-1} \in \ultra{p}$.
Putting $F = \{g_1,\dots,g_n\}$, we have proved that $\{ g \in G : Ag^{-1} \in \ultra{p} \}$ is right syndetic.
Since $A \in \ultra{p}$ the larger set $\{ g \in G : Ag^{-1} \cap A \in \ultra{p} \}$ is also right syndetic.
But $g \in A^{-1} A$ if and only if $Ag^{-1} \cap A$ is non-empty, so $A^{-1}A \supset \{ g \in G : Ag^{-1} \cap A \in \ultra{p} \}$ is right syndetic, as desired.
\end{proof}

A subset of a group $G$ is a \define{right central} set if it belongs to some minimal idempotent ultrafilter, and a \define{right central$^*$} set if it belongs to every minimal idempotent ultrafilter.

\begin{lemma}
\label{lem:central-star-sets-are-syndetic}
Let $G$ be a group and let $A \subset G$ be right central$^*$ set.
Then $A$ is right syndetic.
\end{lemma}
\begin{proof}
Suppose $A$ is not right syndetic.
Then its complement $B$ is right thick, meaning that for every finite subset $F$ of $G$ there is some $h \in G$ such that $hF \subset B$.
This implies that $\mathscr{F} = \{ B g^{-1} : g \in G \}$ is a filter, so there are ultrafilters on $G$ containing $\mathscr{F}$.
The collection $I$ of ultrafilters that contain $\mathscr{F}$ is a closed subset of $\beta G$.
Moreover, it is a right-ideal, for if $\ultra{p} \supset \mathscr{F}$ and $\ultra{q} \in \beta G$ then $\ultra{p} \conv \ultra{q}$ contains $\mathscr{F}$ by \eqref{eqn:ultrafilterConvolution}.
As remarked above, any right ideal contains a minimal right ideal, so there is a minimal idempotent in $I$.
This implies that $B$ is right central, so $A$ is not right central$^*$.
%Then for every finite set $F\subset G$ there exists $g_{F}\in G$ such that $A\cap g_{F}F = \emptyset$.
%Let
%\[
%q := \{ \bigcup_{F\subset G \text{ finite}} g_{F}(F\cap \bigcap_{g\in \tilde F} Fg)
%\ :\tilde F \subset G \text{ finite}\}.
%\]
%Then $q$ is a filter on $G$ consisting of non-empty sets, so it describes a closed subset of $\beta G$.
%By construction that closed subset is closed under multiplication by elements of $G$ on the right, so it is a right ideal in $\beta G$ and therefore it contains a minimal idempotent $\ultra{p}$.
%By construction again, $A\not\in\ultra{p}$.
\end{proof}

\section{Minimally almost periodic groups}
\label{sec:mapGroups}

Let $G$ be any group.
Denote by $\Cb(G)$ the Banach space of all bounded functions $f : G \to \mathbb{C}$ equipped with the supremum norm.
The $G$ actions $\lmult$ and $\rmult$ on $\Cb(G)$, defined by $(\lmult_g f)(x) = f(gx)$ and $(\rmult_g f)(x) = f(xg)$ respectively, are isometric.
A function $f \in \Cb(G)$ is called \define{almost periodic} if the subset $\{ \lmult_g f \,:\, g \in G \}$ of $\Cb(G)$ is relatively compact.
Given a representation $\phi$ of $G$ on a finite-dimensional, complex Hilbert space $V$ and vectors $x,y$ in $V$, the function $f(g) = \langle \phi(g)x, y \rangle$ is almost-periodic.
A group $G$ is \define{minimally almost periodic} if the only almost periodic functions on $G$ are the constant functions.

The following result, a version of \cite[Theorem~2.2]{MR2354320}, will be used repeatedly below.

\begin{theorem}
\label{thm:minimalInvariant}
Let $G$ be a minimally almost periodic group, let $(\mathbf{X},\mu)$ be a measure-preserving system, and let $\ultra{p}$ be a minimal idempotent ultrafilter on $G$.
For any $f$ in $\lp^2(\mathbf{X},\mu)$ we have
\begin{equation}
\label{eqn:minimalProjection}
\lim_{g \to \ultra{p}} T^g f = \condex{f}{\alminv^\mu\mathbf{X}}
\end{equation}
in the weak topology of $\lp^2(\mathbf{X},\mu)$.
\end{theorem}
\begin{proof}
Fix $f \in \lp^2(\mathbf{X},\mu)$.
Equipped with the weak topology, the unit ball of $\lp^2(\mathbf{X},\mu)$ is compact and Hausdorff so the limit in \eqref{eqn:minimalProjection} makes sense via Proposition~\ref{prop:ultrafilterExtension}.
Let $\phi$ be the limit of the sequence $T^g f$ along $\ultra{p}$.
We first show that $\phi$ belongs to $\lp^2(X,\alminv^\mu\mathbf{X},\mu)$.

We claim that the orbit $\{ T^g \phi : g \in G \}$ is relatively compact in the norm topology.
Fix $\epsilon > 0$.
We have
\begin{equation*}
\lim_{h \to \ultra{p}} T^h \phi = \lim_{h \to \ultra{p}} T^h \lim_{g \to \ultra{p}} T^g f = \lim_{g \to \ultra{p} \conv \ultra{p}} T^g f = \phi
\end{equation*}
by Proposition~\ref{prop:ultrafilterAction} because $\ultra{p}$ is idempotent.
Combined with
\begin{equation*}
\nbar T^g \phi - \phi \nbar^2 = \langle T^g \phi, T^g \phi \rangle - \langle T^g \phi, \phi \rangle - \langle \phi, T^g \phi\rangle + \langle \phi, \phi \rangle
\end{equation*}
we see that $A := \{ g \in G : \nbar T^g \phi - \phi \nbar < \epsilon/2 \} = A^{-1}$ belongs to $\ultra{p}$.
Thus $A A^{-1}$ is syndetic by Lemma~\ref{lem:centralSyndeticity}.
Let $F \subset G$ be finite with $A A^{-1} F = G$.
Fix $g \in G$ and write $g = a b^{-1} k$ accordingly.
We see that
\begin{equation*}
\nbar T^g \phi - T^k \phi \nbar = \nbar T^{a b^{-1}} \phi - \phi \nbar = \nbar T^a \phi - T^b \phi \nbar \le \epsilon
\end{equation*}
so the orbit $\{T^g \phi : g \in G \}$ is covered by the balls of radius $\epsilon$ centered at $T^k \phi$ as $k$ runs through $F$.

It follows that for any $\xi \in \lp^2(\mathbf{X},\mu)$ the function $g \mapsto \langle T^g \phi, \xi \rangle$ is almost periodic.
It is therefore constant because $G$ is minimally almost periodic.
Thus $T^g \phi = \phi$ in $\lp^2(\mathbf{X},\mu)$ for every $g \in G$.
Let $\varphi$ be a representative of $\phi$.
We have $T^g \varphi \sim \varphi$ for every $g \in G$, where $\sim$ denotes equality almost everywhere.
Since $\alminv^\mu\mathbf{X}$ contains the measure zero sets, it follows that $\varphi$ is $\alminv^\mu\mathbf{X}$ measurable and that $\phi \in \lp^2(X,\alminv^\mu\mathbf{X},\mu)$.

Lastly, for any $\psi \in \lp^2(X,\alminv^\mu\mathbf{X},\mu)$ we have
\begin{equation*}
\int \phi \cdot \psi \intd\mu = \lim_{g \to \ultra{p}} \int T^g f \cdot \psi \intd\mu = \lim_{g \to \ultra{p}} \int f \cdot (T^g)^{-1} \psi \intd\mu = \int f \cdot \psi \intd\mu
\end{equation*}
so $\phi$ is the orthogonal projection of $f$ on $\lp^2(X,\alminv^\mu\mathbf{X},\mu)$.
\end{proof}

\section{Characteristic factors in sated systems}
\label{sec:chf}

In this section we prove Theorem~\ref{thm:characteristicFactors}.
To do so we need the following construction of a relatively independent self-joining of a measure-preserving system $(\mathbf{X},\mu)$ in $\cat[meas](G^2)$ over $\alminv_2^\mu \mathbf{X}$, by which we mean a measure $\nu$ on $\mathbf{X} \times \mathbf{X}$ satisfying \eqref{eq:cond-prod} below.
Usually (see \cite[Chapter~5]{MR603625}, for example) one would construct such a joining using a disintegration of $\mu$ over $\alminv_2^\mu \mathbf{X}$, but the existence of such a disintegration is not clear when $X$ is non-metrizable.
In our setting, the need for such a disintegration can be circumvented by using limits along minimal idempotent ultrafilters to give an explicit description of the ergodic projection.

\begin{lemma}
\label{lem:cond-prod}
Let $G$ be a minimally almost periodic group and let $(\mathbf{X},\mu)$ be a measure-preserving system in $\cat[meas](G)$.
Then there exists a unique Baire measure $\nu$ on $X \times X$ such that
\begin{equation}
\label{eq:cond-prod}
\int f_1 \otimes f_2 \intd\nu = \int \condex{f_1}{\alminv^\mu\mathbf{X}} \cdot \condex{f_2}{\alminv^\mu\mathbf{X}} \intd\mu
\end{equation}
for any $f_1,f_2 \in \cont(X)$.
\end{lemma}
\begin{proof}
Uniqueness follows immediately by density of $\cont(X) \otimes \cont(X)$ in $\cont(X^2)$, so it remains to show the existence.
To this end fix a minimal idempotent ultrafilter $\ultra{p}$ on $G$.
Since $G$ is minimally almost periodic Theorem~\ref{thm:minimalInvariant} implies that
\begin{equation*}
\lim_{g \to \ultra{p}} T_2^g f = \condex{f}{\alminv^\mu\mathbf{X}}
\end{equation*}
in the weak topology of $\lp^2(\mathbf{X},\mu)$ for every $f \in \lp^2(\mathbf{X},\mu)$.

Let $\delta : X \to X^2$ be the diagonal embedding and let $\lambda$ be the push-forward $\delta \mu$.
Define an action $R$ of $G$ on $X^2$ by $R^g(x_1,x_2) = (x_1,T^g x_2)$.
For any $f_1,f_2 \in \cont(X)$ we have
\begin{equation*}
\lim_{g \to \ultra{p}} \int f_1 \otimes f_2 \intd(R^g \lambda)
=
\lim_{g \to \ultra{p}} \int f_1 \cdot T^g f_2 \intd\mu
=
\int \condex{f_1}{\alminv^\mu\mathbf{X}} \cdot \condex{f_2}{\alminv^\mu\mathbf{X}} \intd\mu
\end{equation*}
by the above.
Since the space of Baire probability measures on $X^2$ is a compact, Hausdorff space the sequence $g \mapsto R^g \lambda$ has a limit along $\ultra{p}$.
Let $\nu$ be this limit.
The above calculation implies that
\begin{equation*}
\int f_1 \otimes f_2 \intd\nu = \int \condex{f_1}{\alminv^\mu\mathbf{X}} \cdot \condex{f_2}{\alminv^\mu\mathbf{X}} \intd\mu
\end{equation*}
for all $f_1,f_2 \in \cont(X)$ as desired.
\end{proof}

Given a measure-preserving system $(\mathbf{X},\mu)$ in $\cat[meas](G^2)$, the measure $\nu$ obtained by applying Lemma~\ref{lem:cond-prod} to the measure-preserving system $(X,T_2,\mu)$ in $\cat[meas](G)$ is called the \define{relatively independent self-joining} of $\mu$ over $\alminv_2^\mu \mathbf{X}$.
It follows immediately from \eqref{eq:cond-prod} and the properties of conditional expectation that $\nu$ is invariant under the commuting $G$ actions $R_1 = T_{1} \times T_{12}$ and $R_2 = T_2 \times I$.
Thus $(X^2,R,\nu)$ is a measure-preserving system in $\cat[meas](G^2)$.
Lastly, writing $\pi_1$ and $\pi_2$ for the coordinate projections $X^2 \to X$, note that \eqref{eq:cond-prod} implies that $\pi_1 \nu = \pi_2 \nu = \mu$ because all three measures agree on $\cont(X)$.

We now turn to the proof of Theorem~\ref{thm:characteristicFactors}, which begins with the following lemma.

\begin{lemma}
\label{lem:char1}
Let $G$ be a discrete, minimally almost periodic group and $\ultra{p}$ a minimal idempotent ultrafilter on $G$.
Let $(\mathbf{X},\mu)$ be a measure-preserving system in $\cat[meas](G^2)$ and let $f_1$ in $\cont(X)$ be bounded by $1$.
Suppose that $(\mathbf{X},\mu)$ is $(\epsilon^2,f_{1},\inv_{1,2})$ sated.
Then
\[
\left| \lim_{g\to \ultra{p}} \int f_{0} \cdot T_{1}^{g} f_{1} \cdot T_{12}^{g} f_{2} \intd\mu - \lim_{g\to \ultra{p}} \int f_{0} \cdot T_{1}^{g} \condex{f_{1}}{\inv_{1,2}\mathbf{X}} \cdot T_{12}^{g} f_{2} \intd\mu \right| < \epsilon.
\]
for any $f_{0},f_{2} \in \cont(X)$ bounded by $1$.
\end{lemma}
\begin{proof}
Define $u : G \to \lp^2(\mathbf{X},\mu)$ by $u(g) = T_1^g \phi \cdot T_{12}^g f_{2}$ where $\phi = f_{1} - \condex{f_{1}}{\inv_{1,2}\mathbf{X}}$.
We have
\begin{equation*}
\lim_{g \to \ultra{p}} \langle u(hg), u(g) \rangle
=
\lim_{g \to \ultra{p}} \int (\phi \cdot T_1^h \phi) \cdot T_2^g (f_{2} \cdot T_{12}^h f_{2}) \intd\mu
=
\int \condex{\phi \cdot T_1^h \phi}{\alminv^\mu_2 \mathbf{X}} \cdot \condex{f_{2} \cdot T_{12}^h f_{2}}{\alminv^\mu_2 \mathbf{X}} \intd\mu
\end{equation*}
for every $h \in G$ by Theorem~\ref{thm:minimalInvariant}.

Let $\nu$ be the relatively independent self-joining of $\mu$ over $\alminv^\mu_2 \mathbf{X}$, which exists by Lemma~\ref{lem:cond-prod}.
Write $\pi_1$ and $\pi_2$ for the coordinate projections $X \times X \to X$.
Define $\mathbf{Y} = (X \times X, R)$ where $R_1 = T_1 \times T_{12}$ and $R_2 = T_2 \times I$.
We know from the above that $(\mathbf{Y},\nu)$ is a $G \times G$ system and that $\pi_1 : (\mathbf{Y},\nu) \to (\mathbf{X},\mu)$ is a factor map.
Another application of Theorem~\ref{thm:minimalInvariant} yields
\begin{equation}
\label{eqn:vdcCalcuation}
\lim_{h \to \ultra{p}} \lim_{g \to \ultra{p}} \langle u(hg), u(g) \rangle
=
\lim_{h \to \ultra{p}} \int (\phi \otimes f_{2}) \cdot R_1^h (\phi \otimes f_{2}) \intd\nu
=
\int \phi \circ \pi_1 \cdot f_{2} \circ \pi_2 \cdot \condex{\phi \otimes f_{2}}{\alminv^\nu_1 \mathbf{Y}} \intd\nu.
\end{equation}
We have $f_{2} \circ \pi_2 \in \cont(\inv_{2}\mathbf{Y})$.
Passing to a further extension $\mathbf{Z}\to\mathbf{Y}$ we may assume that $\condex{\phi \otimes f_{2}}{\alminv^\nu_1 \mathbf{Z}}$ has a representative in $\cont(\inv_1 \mathbf{Z})$.
By definition of $(\epsilon^2,f_{1},\inv_{1,2})$ satedness it follows that \eqref{eqn:vdcCalcuation} is bounded by
\[
\epsilon^{2} \nbar f_{2}\nbar_{\infty} \nbar\phi \otimes f_{2}\nbar_{2}
\leq
\epsilon^{2} \nbar f_{2}\nbar_{\infty} \nbar\phi\nbar_{2} \nbar f_{2}\nbar_{\infty}
\leq
\epsilon^{2},
\]
so
\begin{equation*}
\left| \lim_{g \to \ultra{p}} \int f_{0} \cdot T_1^g(f_{1} - \condex{f_{1}}{\inv_{1,2} \mathbf{X}}) \cdot T_{12}^g f_2 \intd\mu \right| \le \epsilon
\end{equation*}
by Proposition~\ref{prop:vdc}.
\end{proof}

%\begin{lemma}
%\label{cor:char012}
%Let $\mathbf{X}$ be a $G^{2}$ system, $f_{0},f_{2} \in \lp^{\infty}(X,\mu)$, and $f_{1} \in \lp^{\infty}(\alminv^\mu_{1,2}\mathbf{X})$.
%Then
%\[
%\lim_{g\to p} \int f_{0} T_{1}^{g} f_{1} T_{12}^{g} f_{2}
%=
%\lim_{g\to p} \int \condexp(f_{0}|\alminv^\mu_{1,12}\mathbf{X}) T_{1}^{g} \condexp(f_{1}|\alminv^\mu_{1,2}\mathbf{X}) T_{12}^{g} \condexp(f_{2}|\alminv^\mu_{2,12}\mathbf{X}).
%\]
%\end{lemma}

\begin{theorem}
\label{thm:chf2}
Let $G$ be a minimally almost periodic group and let $(\mathbf{X},\mu)$ be a measure-preserving system in $\cat[erg](G^2)$.
For any $\epsilon > 0$ and any $f_1$ in $\lp^\infty(X,\mu)$ bounded by $1$ there is an extension $\pi : (\mathbf{Y},\nu) \to (\mathbf{X},\mu)$ in $\cat[erg](G^2)$ such that
\[
\left| \lim_{g \to \ultra{p}} \int f_0 \cdot T_1^g f_1 \cdot T_{12}^g f_2 \intd\mu - \lim_{g \to \ultra{p}} \int \condex{f_0 \circ \pi}{\alminv^\nu_{1,12} \mathbf{Y}} \cdot T_1^g \condex{f_1 \circ \pi}{\inv_{1,2}\mathbf{Y}} \cdot T_{12}^g \condex{f_2 \circ \pi}{\alminv^\nu_{12,2} \mathbf{Y}} \intd\nu \right| < \epsilon
\]
for all minimal idempotent ultrafilters $\ultra{p}$ on $G$ and all $f_0,f_2$ in $\lp^\infty(X,\mu)$ bounded by $1$.
\end{theorem}
\begin{proof}
Let $(\mathbf{X},\mu)$ be a measure-preserving system in $\cat[erg](G^2)$ and fix $f_1$ in $\lp^\infty(\mathbf{X},\mu)$ bounded by $1$.
Fix $\epsilon > 0$.
By Theorem~\ref{thm:epsilon-sated} and Proposition~\ref{prop:erg-sated} we can find an extension $(\mathbf{Y},\nu)$ of $(\mathbf{X},\mu)$ in $\cat[erg](G^2)$ via a factor map $\pi$ that is $(\epsilon^2,f_1,\inv_{1,2})$ sated in $\cat[meas](G^2)$.
Lemma~\ref{lem:char1} implies that
\begin{equation}
\label{eqn:passingToExtension}
\left| \lim_{g\to \ultra{p}} \int f_{0} \cdot T_{1}^{g} f_{1} \cdot T_{12}^{g} f_{2} \intd\mu - \lim_{g\to \ultra{p}} \int (f_{0} \circ \pi) \cdot T_{1}^{g} \condex{f_{1} \circ \pi}{\inv_{1,2}\mathbf{Y}} \cdot T_{12}^{g} (f_{2} \circ \pi) \intd\nu \right| < \epsilon
\end{equation}
for any $f_0,f_2$ in $\cont(X)$ bounded by $1$.
By density and linearity we may assume $\condex{f_{1} \circ \pi}{\inv_{1,2} \mathbf{Y}} = h_{1}h_{2}$ where $h_{i}$ is $\alminv^\nu_{i}\mathbf{Y}$-measurable.
Under this assumption the second term in \eqref{eqn:passingToExtension} becomes
\[
%\lim_{g\to \ultra{p}} \int (f_{0} \circ \pi) \cdot T_{1}^{g} \condex{f_{1} \circ \pi}{\inv_{1,2}\mathbf{Y}} \cdot T_{12}^{g} f_{2} \intd\mu
\lim_{g\to \ultra{p}} \int (f_{0} \circ \pi) h_{1} \cdot T_{12}^{g} ((f_{2} \circ \pi) h_{2}) \intd\nu
=
\int (f_{0} \circ \pi) h_{1} \cdot \condex{(f_{2} \circ \pi) h_{2}}{\alminv^\nu_{12}\mathbf{Y}} \intd\nu
\]
by Theorem~\ref{thm:minimalInvariant}.
It follows that the above limit vanishes if $f_{0} \circ \pi \perp \alminv^\nu_{1,12}\mathbf{Y}$.
Since conditional expectation is self-adjoint, the above limit also vanishes if  $f_{2} \circ \pi \perp \alminv^\nu_{2,12}\mathbf{Y}$, which gives the desired result.
\end{proof}

\section{Strong Recurrence}
\label{sec:strongRec}

In this section we prove Theorem~\ref{thm:mapCorners}, which follows from Theorem~\ref{thm:mu4} below by passing to a continuous model as follows.
Theorem~\ref{thm:mu4} is a version of \cite[Corollary~4.9]{arxiv:1402.4736} for limits along minimal idempotent ultrafilters.
Given an action $S$ of a group $G$ on a probability space $(X,\mathscr{B},\mu)$ by measurable, measure-preserving maps, consider the space $\mathcal{A}$ of all bounded, measurable functions on $(X,\mathscr{B})$.
Equipped with the supremum norm it becomes a \cstar{}-algebra.
Let $\Omega$ be the Gelfand spectrum of $\mathcal{A}$.
Then $\cont(\Omega)$ and $\mathcal{A}$ are isomorphic as \cstar{}-algebras.
The action $S$ of $G$ on $(X,\mathscr{B},\mu)$ induces an action $T$ of $G$ on $\Omega$ by continuous maps.
Moreover, the measure $\mu$ induces a bounded, linear functional on $\cont(\Omega)$ that can be identified with a Baire probability measure on $\Omega$ that is $T$-invariant.

\begin{theorem}
\label{thm:mu4}
Let $G$ be a minimally almost periodic group and let $(\mathbf{X},\mu)$ be a measure-preserving system in $\cat[erg](G^2)$.
For any minimal idempotent ultrafilter $\ultra{p}$ in $\beta G$ we have
\begin{equation*}
\lim_{g \to \ultra{p}} \int f_{0} \cdot T_1^g f_{1} \cdot T_{12}^g f_{2} \intd\mu
\ge
\left( \int f_{0}^{1/4}f_{1}^{1/4}f_{2}^{1/4} \intd\mu \right)^4
\end{equation*}
for any non-negative measurable functions $f_{0},f_{1},f_{2}$ on $X$.
\end{theorem}
\begin{proof}
By the monotone convergence theorem we may assume that the functions $f_{0},f_{1},f_{2}$ are bounded.
By Theorem~\ref{thm:chf2} it suffices to show that
\[
\lim_{g \to \ultra{p}} \int \condex{f}{\alminv^\mu_{1,12} \mathbf{X}} \cdot T_1^g \condex{f}{\inv_{1,2} \mathbf{X}} \cdot T_{12}^g \condex{f}{\alminv^\mu_{12,2}\mathbf{X}} \intd\mu
\ge
\left( \int f_{0}^{1/4}f_{1}^{1/4}f_{2}^{1/4} \intd\mu \right)^4
\]
for any ergodic measure-preserving system $(\mathbf{X},\mu)$ in $\cat[erg](G^2)$ and any non-negative $f_{0},f_{1},f_{2} \in \lp^{\infty}(X,\mu)$.
We prove that
\begin{equation}
\label{eqn:constantCorrelation}
g \mapsto \int \condex{f_{0}}{\alminv^\mu_{1,12}\mathbf{X}} \cdot T_{1}^{g} \condex{f_{1}}{\inv_{1,2}\mathbf{X}} \cdot T_{12}^{g} \condex{f_{2}}{\alminv^\mu_{12,2}\mathbf{X}} \intd\mu
\end{equation}
does not depend on $g$.
Since any $\alminv_{1,2}^\mu\mathbf{X}$ measurable function can be approximated by linear combinations of functions of the form $\xi_1 \xi_2$ where $\xi_i$ is $\alminv_i^\mu \mathbf{X}$ measurable, we may replace $f_{0}$ above with $h_{0,1}h_{0,12}$, where $h_{j,i}$ is $\alminv^\mu_i \mathbf{X}$ measurable, and similarly $f_{1}$ and $f_{2}$ by $h_{1,1} h_{1,2}$ and $h_{2,12} h_{2,2}$ respectively.
Then the above integral equals
\[
\int h_{0,12} h_{2,12} \cdot h_{0,1} h_{1,1} \cdot T_{12}^{g} (h_{1,2} h_{2,2}) \intd\mu
\]
and it therefore suffices to show that the sub-$\sigma$-algebras $\alminv^\mu_{1} \mathbf{X}$, $\alminv^\mu_{2} \mathbf{X}$, and $\alminv^\mu_{12} \mathbf{X}$ are jointly independent.
Let now $h_{i}\in \lp^\infty(X,\mu)$ be $\alminv^\mu_i\mathbf{X}$ measurable for $i=1,2,12$.
We have to show
\[
\int h_{1} \cdot h_{2} \cdot h_{12} \intd\mu = \int h_{1} \intd\mu \int h_{2} \intd\mu \int h_{12} \intd\mu.
\]
From
\[
\int h_1 \cdot h_2 \cdot h_{12} \intd \mu
=
\int \condex{h_1}{\alminv_1^\mu \mathbf{X}} \cdot \condex{h_2}{\alminv_2^\mu \mathbf{X}} \cdot h_{12} \intd \mu
=
\int h_1 \cdot h_2 \cdot \condex{h_{12}}{\alminv_{1,2}^\mu \mathbf{X}} \intd \mu
\]
and the fact that the conditional expectation onto $\alminv_{1,2}^\mu \mathbf{X}$ commutes with $T_{12}$ (since this $\sigma$-algebra is $T_{12}$-invariant), we may additionally assume that $h_{12}$ is $\alminv^\mu_{1,2} \mathbf{X}$ measurable.

Ergodicity of $(\mathbf{X},\mu)$ implies the sub-$\sigma$-algebras $\alminv^\mu_{1} \mathbf{X}$ and $\alminv^\mu_{2} \mathbf{X}$ are independent.
Thus
\[
\lp^2(\alminv_{1,2}^\mu \mathbf{X}) \cong \lp^2(\alminv_1^\mu \mathbf{X}) \otimes \lp^2(\alminv_2^\mu \mathbf{X})
\]
and, thinking of $h_{12}$ as an element of the right-hand side, we see that $h_{12}$ is a member of
\[
\lp^2(\mathsf{C}_2^\mu \alminv_1^\mu \mathbf{X}) \otimes \lp^2(\mathsf{C}_1^\mu \alminv_2^\mu \mathbf{X})
\]
by Lemma~\ref{lem:jdgSplitting}.
(Here $\mathsf{C}_i^\mu$ denotes the $\sigma$-algebra of $T_i$ almost periodic functions.)
Since $G$ is minimally almost periodic $\mathsf{C}_2 \alminv_1^\mu \mathbf{X}$ and $\mathsf{C}_1 \alminv_2^\mu \mathbf{X}$ are both the trivial $\sigma$-algebra.
This proves \eqref{eqn:constantCorrelation} is constant.
Finally
\begin{align*}
&
\lim_{g \to \ultra{p}} \int \condex{f_{0}}{\alminv^\mu_{1,12}\mathbf{X}} \cdot T_1^g \condex{f_{1}}{\inv_{1,2}\mathbf{X}} \cdot T_{12}^g \condex{f_{2}}{\alminv^\mu_{12,2} \mathbf{X}} \intd\mu
\\
=
&
\int \condex{f_{0}}{\alminv^\mu_{1,12}\mathbf{X}} \cdot \condex{f_{1}}{\inv_{1,2}\mathbf{X}} \cdot \condex{f_{2}}{\alminv^\mu_{12,2} \mathbf{X}} \intd\mu
\\
\ge
&
\left( \int f_{0}^{1/4}f_{1}^{1/4}f_{2}^{1/4} \intd\mu \right)^4
\end{align*}
by Lemma~\ref{lem:chu}.
\end{proof}

\section{Combinatorial results}
\label{sec:combinatorics}

We begin by proving Theorem~\ref{thm:syndetic}.
Fix an ultrafilter on $\mathbb{N}$ for taking ultraproducts.
Let $n \mapsto G_n$ be a quasirandom sequence of finite groups.
Let $G$ be the ultraproduct of the sequence $n \mapsto G_n$ and let $\Omega$ be the ultraproduct of the sequence $n \mapsto G_n \times G_n$.
We consider the commuting actions $L_1$ and $L_2$ of $G$ on $\Omega$ defined by $L_1^g (x,y) = (gx,y)$ and $L_2^g(x,y) = (x,gy)$ respectively.
The induced $G \times G$ action $L$ is just the action of $\Omega$ on itself by left multiplication.
We first note that this action is ergodic with respect to the Loeb measure $\haar$ on $\Omega$ provided $G$ is minimally almost periodic.

\begin{lemma}
If $G$ is minimally almost periodic then the $G \times G$ action $L$ on $(\Omega,\haar)$ is ergodic.
\end{lemma}
\begin{proof}
Given a Loeb measurable subset $B$ of $\Omega$ we have $\haar(B \cap (L^x)^{-1} B) = \haar(B)^2$ for Loeb almost every $x \in \Omega$ by \cite[Lemma~33]{MR3177376}.
Thus if $B$ is almost invariant then $\haar(B) \in \{0,1\}$.
\end{proof}

In fact \cite[Lemma~33]{MR3177376} implies $L$ is weak mixing, but we will not need this.

It will be convenient later to pass to a model of this action on a compact, Hausdorff space.
We do so by considering the C$^*$ algebra of bounded, measurable functions on $\Omega$, which can be represented as $\cont(X)$ for some compact, Hausdorff topological space $X$.
Let $\mathscr{B}$ be the Baire sub-$\sigma$-algebra of $X$.
The Loeb measure $\haar$ on $\Omega$ passes to a probability measure $\mu$ on $(X,\mathscr{B})$ and the $G$ actions $L_1$ and $L_2$ become actions $T_1$ and $T_2$ of $G$ on $X$ by homeomorphisms.
Write $T$ for the induced $G \times G$ action $(g_1,g_2) \mapsto T^{(g_1,g_2)}$ on $X$.
Thus we have a measure-preserving system $(\mathbf{X},\mu)$ in $\cat[meas](G^2)$.
This system is ergodic when $G$ is minimally almost periodic.

\begin{proposition}
If $G$ is minimally almost periodic then $(\mathbf{X},\mu)$ is ergodic.
\end{proposition}
\begin{proof}
Fix a continuous function $f$ on $X$ and a minimal idempotent ultrafilter $\ultra{p}$ on $G \times G$.
By Theorem~\ref{thm:minimalInvariant} we have
\begin{equation*}
\lim_{g \to \ultra{p}} \int \phi \cdot  T^g f \intd\mu = \int \phi \cdot \condex{f}{\alminv^\mu\mathbf{X}} \intd\mu
\end{equation*}
for any continuous $\phi$.
By evaluating the left hand side on $\Omega$ rather than on $X$, we obtain
\begin{equation*}
\lim_{g \to \ultra{p}} \int \phi \cdot T^g f \intd\mu = \int \phi \intd\mu \int f \intd\mu
\end{equation*}
by Theorem~\ref{thm:minimalInvariant} and the previous lemma.
Thus $\condex{f}{\alminv^\mu \mathbf{X}} = \int f \intd\mu$.
The same is true for any function that can be approximated in $\lp^2(X,\mathscr{B},\mu)$ by continuous functions.
\end{proof}

%Given $f$ in $\lp^2(X,\mu)$, we say that $\condex{f}{\alminv^\mu_2\mathbf{X}}$ is \define{continuous} if it is representable by a continuous function.

%\begin{proposition}
%If $G$ is minimally almost periodic then $\condex{f}{\alminv^\mu_2\mathbf{X}}$ is continuous for every $f \in \cont(X)$.
%\end{proposition}
%\begin{proof}
%Let $\ultra{p}$ be a minimal idempotent ultrafilter on $G$.
%For any $\phi,f \in \cont(X)$ we have
%\begin{equation*}
%\lim_{g \to \ultra{p}} \int \phi \cdot T_2^g f \intd\mu = \int \phi \cdot \condex{f}{\alminv^\mu_2\mathbf{X}} \intd\mu
%\end{equation*}
%by Theorem~\ref{thm:minimalInvariant}.
%Let $\Phi$ and $F$ be the bounded, measurable functions on $\Omega$ that correspond to $\phi$ and $f$ respectively, and let $\psi$ be a continuous function on $X$ corresponding to $\lim_{g \to \ultra{p}} L^g F$.
%Another application of Theorem~\ref{thm:minimalInvariant} yields
%\begin{equation*}
%\lim_{g \to \ultra{p}} \int \Phi \cdot L^g F \intd\haar = \int \phi \cdot \psi \intd\mu
%\end{equation*}
%and, since $\phi \in \cont(X)$ was arbitrary, we have $\condex{f}{\alminv^\mu_2\mathbf{X}} = \psi$ in $\lp^2(\mathbf{X})$.
%\end{proof}

Now we turn to the proof of Theorem~\ref{thm:syndetic}.

\begin{proof}[Proof of Theorem~\ref{thm:syndetic}]
Suppose that the conclusion fails for some $0 < \alpha < 1$ and some $\epsilon > 0$.
Then we can find sequences $D_n \to \infty$ and $K_n \to \infty$ in $\mathbb{N}$, a sequence $n \mapsto G_n$ of $D_{n}$-quasirandom groups in $\mathcal{F}$, and sets $A_{n} \subset G_{n}\times G_{n}$ with $|A_{n}|\geq\alpha |G_{n}|^{2}$ such that
\[
R_{n} := \left\{ g \in G_{n} : \frac{|A_{n}\cap (1,g)^{-1}A_{n} \cap (g,g)^{-1}A_{n}|}{|G_n|^2} > \alpha^{4}-\epsilon \right\}
\]
is not right $K_{n}$-syndetic.
Since being right $K$-syndetic for some finite $K$ is a first order property, it follows that the ultraproduct $R$ of the sequence $n \mapsto R_n$ is not right $K$-syndetic for any $K$, and therefore not right syndetic.

Let $G$ be the ultraproduct of the sequence $n \mapsto G_n$ and let $\Omega$ be the ultraproduct of the sequence $n \mapsto G_n \times G_n$.
Since $\mathcal{F}$ is a quasirandom ultraproduct class the group $G$ is minimally almost periodic.
Let $A$ be the internal subset of $\Omega$ determined by the sequence $n \mapsto A_n$.
We have
\[
R
\supseteq
\{g\in G : \mu(A \cap (1,g)^{-1}A \cap (g,g)^{-1}A) > \alpha^{4}-\epsilon/2\}
\]
where $\mu$ is the Loeb measure on $\Omega$.
But $R$ is right central$^*$ by Theorem~\ref{thm:mu4}, and therefore right syndetic by Lemma~\ref{lem:central-star-sets-are-syndetic}, giving the desired contradiction.
\end{proof}

\begin{proof}[Proof of Theorem~\ref{thm:increasingGroups}]
Fix a quasirandom sequence $n \mapsto G_n$ with $G_n \incl G_{n+1}$ for all $n \in \mathbb{N}$.
Suppose the theorem is false for some $0 < \alpha < 1$ and some $\epsilon > 0$.
Then we have a sequence of sets $A_n \subset G_n \times G_n$ with $|A_n| \ge \alpha |G_n|^2$ and
\begin{equation}
\label{eqn:increasingReturns}
| \{ g \in G_n : A_n \cap (1,g)^{-1}A_n \cap (g,g)^{-1} A_n \ne \emptyset \} | \le (1 - \epsilon) |G_n|
\end{equation}
for all $n \in \mathbb{N}$.

Let $G$ be the direct limit of the sequence $n \mapsto G_n$.
Put $A = \cup \{ A_n : n \in \mathbb{N} \}$ in $G \times G$.
We have
\[
\limsup_{N \to \infty} \frac{|A \cap G_N \times G_N|}{|G_N \times G_N|} \ge \alpha
\]
so $A$ has positive upper density with respect to the F\o{}lner sequence $N \mapsto G_N \times G_N$ in $G \times G$.
By \cite[Lemma~5.2]{arxiv:1309.6095} there is an ergodic action $T$ of $G \times G$ on a compact, metric probability space $(X,\mathscr{B},\mu)$, an open set $U \subset X$ with $\mu(U) = \upperdens_\Phi(A)$, and a F\o{}lner sequence $\Psi$ on $G \times G$ such that
\[
\upperdens_\Psi((g_1,h_1)^{-1} A \cap \cdots \cap (g_n,h_n)^{-1} A)
\ge
\mu((T^{(g_1,h_1)})^{-1} U \cap \cdots \cap (T^{(g_n,h_n)})^{-1} U)
\]
for all $g_n,h_n \in G$.
In particular
\[
\upperdens_\Psi(A \cap (1,g)^{-1}A \cap (g,g)^{-1}A)
\ge
\mu(U \cap (T_1^g)^{-1} U \cap (T_1^g T_2^g)^{-1}U)
\]
for all $g \in G$.
It follows from \cite[Corollary~4.9]{arxiv:1402.4736} that for every $\epsilon > 0$ the set
\[
\{ g \in G : \upperdens_\Psi(A \cap (1,g)^{-1}A \cap (g,g)^{-1}A) > 0 \}
\]
has full density with respect to every F\o{}lner sequence in $G$.
In particular, it has full density with respect to the F\o{}lner sequence $N \mapsto G_N \times G_N$, contradicting \eqref{eqn:increasingReturns} for $n$ large enough.
\end{proof}

\appendix
\section{Chu's inequality}
We use the following slightly generalized version of Chu's lower bound for a product of conditional expectations \cite[Lemma 1.6]{MR2794947}.
\begin{lemma}
\label{lem:chu}
Let $f_{0},\dots,f_{n}$ be non-negative integrable functions on a probability space $(X,B,\mu)$ and let $B_{1},\dots,B_{n}\subset B$ be arbitrary sub-$\sigma$-algebras.
Then
\[
\int f_{0} \prod_{i=1}^{n} \condex{f_{i}}{B_{i}}
\geq
\Big( \int \prod_{i=0}^{n} f_{i}^{\frac{1}{n+1}} \Big)^{n+1}.
\]
\end{lemma}
The main advantage of the present formulation is the ability to take $f_{0}\equiv 1$.
\begin{proof}
Note that $\{f_{i}>0\} \subset \{\condex{f_{i}}{B_{i}} > 0\}$ up to a set of measure zero for every $i$.
By H\"older's inequality we have
\begin{align*}
\int \prod_{i=0}^{n} f_{i}^{\frac{1}{n+1}}
&=
\int \big( f_{0}^{\frac{1}{n+1}} \prod_{i=1}^{n} \condex{f_{i}}{B_{i}}^{\frac{1}{n+1}} \big)
\cdot \prod_{i=1}^{n} \Big( 1_{\{\condex{f_{i}}{B_{i}} > 0\}} \frac{f_{i}}{\condex{f_{i}}{B_{i}}} \Big)^{\frac{1}{n+1}}\\
&\leq \Big(\int f_{0} \prod_{i=1}^{n} \condex{f_{i}}{B_{i}} \Big)^{\frac{1}{n+1}}
\cdot \prod_{i=1}^{n} \Big( \int_{\{\condex{f_{i}}{B_{i}} > 0\}} \frac{f_{i}}{\condex{f_{i}}{B_{i}}} \Big)^{\frac{1}{n+1}}.\\
\intertext{Since the functions $\condex{f_{i}}{B_{i}}^{-1}$ are $B_{i}$-measurable, $f_{i}$'s may be replaced by their expectations onto $B_{i}$ in the integrals in the second factor.
Thus we obtain}
&=
\Big(\int f_{0} \prod_{i=1}^{n} \condex{f_{i}}{B_{i}} \Big)^{\frac{1}{n+1}}
\cdot \prod_{i=1}^{n} |\{\condex{f_{i}}{B_{i}} > 0\}|^{\frac{1}{n+1}}\\
&\leq
\Big(\int f_{0} \prod_{i=1}^{n} \condex{f_{i}}{B_{i}} \Big)^{\frac{1}{n+1}}.
\qedhere
\end{align*}
\end{proof}
It would be interesting to know whether the lower bound in Lemma~\ref{lem:chu} is sharp for some characteristic functions $f_{i}$.
There are two sources of inefficiency in its proof: the H\"older inequality and the estimate $|\{\condex{f_{i}}{B_{i}}\}|\leq 1$.
It is clear that the second source of inefficiency can be easily eliminated.
On the other hand, an example in which the H\"older inequality gives a sharp estimate can be found in \cite[Appendix B]{arxiv:1309.6095} (in hindsight this provides an explanation for why that example, which has been initially found numerically, works).
However, it is not clear whether both sources of inefficiency can be controlled simultaneously.

\printbibliography

\end{document}